\documentclass{amsart}
\usepackage{graphicx}
\usepackage{hyperref}

\newtheorem{theorem}{Theorem}[section]
\newtheorem{lemma}[theorem]{Lemma}
\newtheorem{conjecture}[theorem]{Conjecture}
\newtheorem{proposition}[theorem]{Proposition}
\theoremstyle{definition}
\newtheorem{definition}[theorem]{Definition}

\theoremstyle{remark}
\newtheorem{remark}[theorem]{Remark}

\numberwithin{equation}{section}

\begin{document}

\title{On a conjecture of Laugesen and Morpurgo}

\author{Mihai N. Pascu}
\address{Faculty of Mathematics and Computer Science, Transilvania University of Bra\c{s}ov, Bra\c{s}ov -- 500091, ROMANIA}
\email{mihai.pascu@unitbv.ro}
\thanks{Supported by CNCSIS grant PNII - ID 209.}

\author{Maria E. Gageonea}
\address{Faculty of Mathematics and Computer Science, Transilvania University of Bra\c{s}ov, Bra\c{s}ov -- 500091, ROMANIA}
\email{gageonea@yahoo.com}
\thanks{Supported by CNCSIS grant PNII - ID 209.}




\maketitle

\begin{abstract}
A well known conjecture of R. Laugesen and C. Morpurgo asserts that
the diagonal element of the Neumann heat kernel of the unit ball in
$\mathbb{R}^{n}$ ($n\geq 1$) is a radially increasing function. In
this paper, we use probabilistic arguments to settle this
conjecture, and, as an application, we derive a new proof of the Hot
Spots conjecture of J. Rauch in the case of the unit disk.
\end{abstract}

\section{Introduction}

We learned from Rodrigo Ba\~{n}uelos the following conjecture of
Richard Laugesen and Carlo Morpurgo which arose in connection with
their work on conformal extremals of zeta functions of eigenvalues
under Neumann boundary conditions in \cite{Laugesen-Morpurgo}:

\begin{conjecture}[Laugesen-Morpurgo Conjecture]
\label{Laugesen-Morpugo conjecture in R^2}Let $p_{U}(t,x,y)$ denote
the heat kernel for the Laplacian with Neumann boundary conditions
on the unit disk $U=\left\{ z\in \mathbb{C}:\left\vert z\right\vert
<1\right\} $ in $\mathbb{C} $. For any $t>0$ arbitrarily fixed, the
radial function $p_{U}(t,x,x)$, called the diagonal element of the
heat kernel, is a strictly increasing function of $|x|$, that is,
\begin{equation}
p_{U}(t,x,x)<p_{U}(t,y,y),  \label{Laugesen conjecture}
\end{equation}%
for any $t>0$ and $x,y\in U$ with $\left\vert x\right\vert
<\left\vert y\right\vert $.
\end{conjecture}

Surprisingly, despite the seemingly simple nature of this conjecture
and the fact that it seems to have been well known since 1994, we do
not know of any
progress on it, aside from some partial related results (see \cite%
{PascuNicolaie}, \cite{PascuPascu} and \cite{PascuPascu1}). A more
recent
result related to this conjecture is due to Ba\~{n}uelos et. all (\cite%
{Banuelos}), in which the authors show that the Laugesen-Morpurgo
conjecture (\ref{Laugesen conjecture}) holds for the transition
density of the $n$ -dimensional Bessel processes on $(0,1]$
reflected at $1$ in the case $n>2$, and that this is false for
$n=2$. Since the absolute value of a $n$ -dimensional Brownian
motion is a Bessel process of order $n$, this is equivalent to the
monotonicity with respect to $x\in \left( 0,1\right) $ of the integral mean%
\begin{equation*}
\int_{0}^{2\pi }p_U\left( t,x,xe^{i\theta }\right) d\theta .
\end{equation*}

\begin{remark}
The Laugesen--Morpurgo conjecture has the following physical
interpretation. Consider a thermally insulated planar disk in which
an atom of heat has been introduced at time $t=0$. Then one feels
\textquotedblleft warmest\textquotedblright\ for all $t>0$ at the
point where the atom of heat has been introduced if the point was
chosen on the boundary of the disk, and feels \textquotedblleft
coldest\textquotedblright\ at this point for all $t>0 $ if the
chosen point was the center of the disk. Moreover, the corresponding
temperature function (measured at the point where the atom of heat
was inserted) is, for all $t>0$ arbitrarily fixed, a radially
increasing function with respect to the point where the atom of heat
was introduced. See the Theorem \ref{Hot Spots for the unit disk}
and Remark \ref {Hot Spots remark} for a connection of the
Laugesen-Morpurgo conjecture with the Hot Spots conjecture of
Jeffrey Rauch.
\end{remark}

In this paper, we use probabilistic arguments (couplings of
reflecting Brownian motions) to settle the Laugesen-Morpurgo
conjecture, first in the $2 $-dimensional case (Section \ref{Main
results}), then in the general case (Section \ref{Extensions and
Applications}). The paper is organized as follows: in Section
\ref{Preliminaries} we present the mirror coupling introduced by
Burdzy el. all (\cite{BurdzyKendall}, and more recently
\cite{AtarBurdzy}, \cite{AtarBurdzy2}), we establish the notation
and we derive some properties of the coupling needed for the proof
in the particular case of the reflecting Brownian motions in the
unit disk.

In Section \ref{Main results},  in the particular case of the unit
disk, we describe the motion of the mirror of the coupling (i.e. the
line of the symmetry between the two processes), more precisely we
show that for a certain choice of the starting points, the mirror
always moves away from its starting point, towards the origin. This
geometric property of the coupling, together with some probabilistic
and analytic arguments allows us to settle the Laugesen-Morpurgo
conjecture in the $2$-dimensional case (Theorem \ref {main
theorem}).

In Section \ref{Extensions and Applications} we show that the
arguments used in the previous section can be extended to obtain a
proof of the general Laugesen-Morpurgo conjecture (we first present
the $1$-dimensional case, then we give the proof of the general case
for $n=3$, since it is easier to follow geometrically and
notation-wise, yet it contains all the ideas involved in the proof
of the general case), in Theorem \ref{general laugesen-Morpurgo
conjecture}.

As an application, we derive a different proof of the Hot Spots
conjecture of Jeffrey Rauch, which suggests a possibly different
approach for a resolution of this later conjecture, solved only
partially at the present moment.

\section{Preliminaries\label{Preliminaries}}

Our proof of Laugesen-Morpurgo conjecture relies on a certain
property of the mirror coupling of reflecting Brownian motions in
the unit disk and a representation of the Neumann heat kernel as an
occupation time density of reflecting Brownian motion. We begin with
a presentation of these results.

We denote by $U=\left\{ z\in \mathbb{C}:\left\vert z\right\vert
<1\right\} $ the unit disk in $\mathbb{C}$ (which we identify with
$\mathbb{R}^{2}$).

We define the reflecting Brownian motion in $U$ as a solution of the
stochastic differential equation:
\begin{equation}
X_{t}=X_{0}+B_{t}+\int\limits_{0}^{t}\nu (X_{s})dL_{s},\qquad t\geq
0\text{.}  \label{Skorokhod equation}
\end{equation}

Formally we have:

\begin{definition}
\label{Definition of RBM}$X_{t}$ is a reflecting Brownian motion in
$U$ starting at $x_{0}\in \overline{U}$ if it satisfies
(\ref{Skorokhod equation}), where:

(a) $B_{t}$ is a $2$-dimensional Brownian motion started at $0$ with
respect to a filtered probability space $(\Omega ,\mathcal F,
(\mathcal F_t)_{t\geq 0}, P)$,

(b) $L_{t}$ is a continuous nondecreasing process which increases
only when $X_{t}\in \partial U$, i.e. $\int_0^{\infty} 1_U(X_t)\,
dL_t = 0$, a.s.

(c) $X_{t}$ is $(\mathcal{F}_{t})$-adapted, and almost surely
$X_{0}=x_{0}$ and $X_{t}\in \overline{U}$ for all $t\geq 0$.
\end{definition}

\begin{remark}
For pathwise existence and uniqueness of reflecting Brownian motion
in the sense of the above definition see for example
\cite{Bass-Hsu2}.
\end{remark}

Krzysztof Burdzy et. all (\cite{BurdzyKendall}, and more recently
\cite{AtarBurdzy}, \cite{AtarBurdzy2}) introduced the mirror
coupling of reflecting Brownian motions $X_{t}$ and $Y_{t}$ in a
smooth planar domain $D\subset \mathbb{R}^{2}$, which we will
describe briefly below.

The ideea of the coupling is that the two processes behave like
ordinary Brownian motion (symmetric with respect to a line of
symmetry, called the mirror of the coupling) when both of them are
inside the domain $D$. When one of the processes hits the boundary,
the mirror $M_{t}$ gets a (minimal) push towards the inward unit
normal at the correponding point at the boundary, needed in order to
keep both processes in $\bar{D}$.

Considering the coupling time $\tau =\inf \left\{ t\geq 0:X_{t}\in
M_{t}\right\} $, the mirror coupling evolves as described above for
$t < \tau $, and we let $X_{t}=Y_{t}$ for $t\geq \tau $ (the two
processes move together after the coupling time). For definiteness,
for $t\geq \tau $ we define the mirror $M_{t}$ as the line passing
through $X_{t}=Y_{t}$ and making the same angle as $M_{\tau }$ with
the horizontal axis.

Formally, given two arbitrarily fixed points $x,y\in U$, we define
the mirror coupling of reflecting Brownian motions in the unit disk,
as a pair $\left( X_{t},Y_{t}\right) _{t\geq 0}$ of stochastic
processes on a common filtered probability space $(\Omega, \mathcal
F,(\mathcal F_t)_{t\geq0},P)$, given by
\begin{equation}
\left\{
\begin{array}{l}
X_{t}=x+W_{t}+\int_{0}^{t}\nu \left( X_{s}\right) dL_{s}^{X} \\
Y_{t}=y+Z_{t}+\int_{0}^{t}\nu \left( Y_{s}\right) dL_{s}^{Y}%
\end{array}%
\right.  \label{mirror coupling in U}
\end{equation}%
where $W_{t}$ is a $2$-dimensional Brownian motion starting at
$W_{0}=0$, $(\mathcal F_t)_{t\geq 0}$ is the filtration generated by
the Brownian motion $W_t$, $Z_{t}$ is the mirror image of the
Brownian motion $W_{t}$ with respect to the line of symmetry $M_{t}$
between $X_{t}$ and $Y_{t}$, that is
\begin{equation}
Z_{t}=W_{t}-2\int_{0}^{t}\frac{X_{s}-Y_{s}}{\left\Vert
X_{s}-Y_{s}\right\Vert ^{2}}\left( X_{s}-Y_{s}\right) \cdot dW_{s},
\label{definition of Z_t}
\end{equation}%
$L_{t}^{X}$ and $L_{t}^{Y}$ denote the boundary local times of the
reflecting Brownian motions $X_{t}$ and respectively $Y_{t}$, and
$\nu \left( z\right) =-z$, $z\in \partial U$, denotes the inward
unit normal vector field on the boundary of $U$. The processes
$X_{t}$ and $Y_{t}$ evolve according to (\ref{mirror coupling in U})
above for $t < \tau $,
where $\tau $ is the coupling time (possibly infinite)%
\begin{equation*}
\tau =\inf \left\{ t>0:X_{t}=Y_{t}\right\} \in \mathbb{R}\cup
\left\{ \infty \right\} ,
\end{equation*}%
and the two processes evolve together (i.e. $X_{t}=Y_{t}$) after
they have coupled.

Setting%
\begin{equation}
\left\{
\begin{array}{c}
X_{t}-Y_{t}=\left( m_{t},n_{t}\right) \\
X_{t}+Y_{t}=\left( p_{t},q_{t}\right)%
\end{array}%
\right. ,\qquad t\geq 0,  \label{processes M N P Q}
\end{equation}%
we have%
\begin{equation*}
\left\{
\begin{array}{c}
m_{t}=x^{1}-y^{1}+W_{t}^{1}-Z_{t}^{1}-\int_{0}^{t}X_{s}^{1}dL_{s}^{X}+%
\int_{0}^{t}Y_{s}^{1}dL_{s}^{Y} \\
n_{t}=x^{2}-y^{2}+W_{t}^{2}-Z_{t}^{2}-\int_{0}^{t}X_{s}^{2}dL_{s}^{X}+%
\int_{0}^{t}Y_{s}^{2}dL_{s}^{Y} \\
p_{t}=x^{1}+y^{1}+W_{t}^{1}+Z_{t}^{1}-\int_{0}^{t}X_{s}^{1}dL_{s}^{X}-%
\int_{0}^{t}Y_{s}^{1}dL_{s}^{Y} \\
q_{t}=x^{2}+y^{2}+W_{t}^{2}-Z_{t}^{2}-\int_{0}^{t}X_{s}^{2}dL_{s}^{X}-%
\int_{0}^{t}Y_{s}^{2}dL_{s}^{Y}%
\end{array}%
\right. ,
\end{equation*}%
for all $t < \tau $, where the superscript $1$ or $2$ indicates the
first, respectively the second component of the given point (when
necessary, in order to avoid confusion with powers, we will use
parantheses in order to indicate the square of a number).

Using the definition (\ref{definition of Z_t}) of $Z_{t}$, we obtain%
\begin{equation}
m_{t}=x^{1}-y^{1}+2\int_{0}^{t}\frac{m_{s}}{m_{s}^{2}+n_{s}^{2}}\left(
m_{s}dW_{s}^{1}+n_{s}dW_{s}^{2}\right)
-\int_{0}^{t}X_{s}^{1}dL_{s}^{X}+\int_{0}^{t}Y_{s}^{1}dL_{s}^{Y},
\label{process m}
\end{equation}%
and therefore the quadratic variation of $M_{t}$ is given by%
\begin{equation*}
\left\langle m\right\rangle _{t}=4\int_{0}^{t}\frac{m_{s}^{2}}{%
m_{s}^{2}+n_{s}^{2}}ds.
\end{equation*}

Similarly, we can obtain the following
\begin{equation}
\left\{
\begin{array}{l}
n_{t}=x^{2}-y^{2}+2\int_{0}^{t}\frac{n_{s}}{m_{s}^{2}+n_{s}^{2}}\left(
m_{s}dW_{s}^{1}+n_{s}dW_{s}^{2}\right)
-\int_{0}^{t}X_{s}^{2}dL_{s}^{X}+\int_{0}^{t}Y_{s}^{2}dL_{s}^{Y} \\
p_{t}=x^{1}+y^{1}+2\int_{0}^{t}\frac{n_{s}}{m_{s}^{2}+n_{s}^{2}}\left(
n_{s}dW_{s}^{1}-m_{s}dW_{s}^{2}\right)
-\int_{0}^{t}X_{s}^{1}dL_{s}^{X}-\int_{0}^{t}Y_{s}^{1}dL_{s}^{Y} \\
q_{t}=x^{2}+y^{2}+2\int_{0}^{t}\frac{n_{s}}{m_{s}^{2}+n_{s}^{2}}\left(
-n_{s}dW_{s}^{1}+m_{s}dW_{s}^{2}\right)
-\int_{0}^{t}X_{s}^{1}dL_{s}^{X}-\int_{0}^{t}Y_{s}^{1}dL_{s}^{Y}%
\end{array}%
\right. ,  \label{processes n, p, q}
\end{equation}%
and the corresponding quadratic variation processes:%
\begin{equation}\label{quadratic variations of m, n, p, q}
\left\{
\begin{split}
&\langle m\rangle _{t}=\langle q\rangle _{t}=4\int_{0}^{t}\frac{m_{s}^{2}}{%
m_{s}^{2}+n_{s}^{2}}ds \\
&\langle n\rangle _{t}=\langle p\rangle _{t}=4\int_{0}^{t}\frac{n_{s}^{2}}{%
m_{s}^{2}+n_{s}^{2}}ds \\
&\langle m,n\rangle _{t}=-\langle p,q\rangle _{t}=4\int_{0}^{t}\frac{%
m_{s}n_{s}}{m_{s}^{2}+n_{s}^{2}}ds \\
&\langle m,p\rangle _{t}=\langle m,q\rangle _{t}=\langle n,p\rangle
_{t}=\langle n,q\rangle _{t}=0%
\end{split}
\right.
\end{equation}

\section{Main results\label{Main results}}

For $t<\tau $, the equation of the line of symmetry $M_{t}$ between
$X_{t}$ and $Y_{t}$ is given by
\begin{equation*}
\left( z-\frac{X_{t}+Y_{t}}{2}\right) \cdot \left(
X_{t}-Y_{t}\right) =0,
\end{equation*}%
or equivalent%
\begin{equation*}
m_{t}u+n_{t}v-\frac{1}{2}\left( m_{t}p_{t}+n_{t}q_{t}\right)
=0\text{,}
\end{equation*}%
where $z=\left( u,v\right) $.

\begin{figure}[!ht]
\centering
\includegraphics*[width=2.1715in]{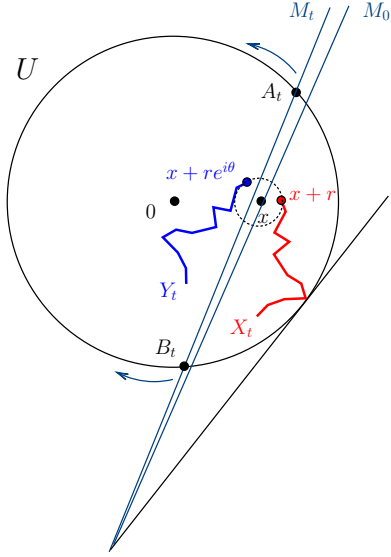}
\caption{The mirror coupling of reflecting Brownian motions in $U$.}
\label{mirror coupling figure}
\end{figure}

The intersection of the mirror $M_{t}$ with the boundary of $U$
consists of two points $A_{t}=\left( a_{t}^{1},a_{t}^{2}\right) $
and $B_{t}=\left( b_{t}^{1},b_{t}^{2}\right) $. The ideea of the
proof is that the mirror $ M_{t}$ moves to the left, in such a way
that $a_{t}^{1}$ and $b_{t}^{1}$ are always decreasing (see Figure
\ref{mirror coupling figure}).

To show this, we consider the stopping time $\tau _{1}=\inf \left\{
t>0:0\in M_{t}\right\} $, and we consider the processes

\begin{equation}
\left\{
\begin{array}{c}
u_{t}=\dfrac{2m_{t}}{m_{t}p_{t}+n_{t}q_{t}} \\
\\
v_{t}=\dfrac{2n_{t}}{m_{t}p_{t}+n_{t}q_{t}}%
\end{array}%
\right. ,\qquad t<\tau \wedge \tau _{1}.
\end{equation}

Note that for $t<\tau \wedge \tau _{1}$ we have%
\begin{equation*}
m_{t}p_{t}+n_{t}q_{t}=\left( X_{t}-Y_{t}\right) \cdot \left(
X_{t}+Y_{t}\right) =\left\vert X_{t}\right\vert ^{2}-\left\vert
Y_{t}\right\vert ^{2}\neq 0
\end{equation*}%
and also%
\begin{equation*}
m_{t}^{2}+n_{t}^{2}=\left\Vert X_{t}-Y_{t}\right\Vert ^{2}\neq
0\text{,}
\end{equation*}%
so the formulae above are well defined.

\begin{lemma}
\label{u and v proc of bdd var}For $t<\tau \wedge \tau _{1}$,
$u_{t}$ and $v_{t}$ defined above are processes of bounded variation, explictly given by%
\begin{equation*}
du_{t}=\frac{2}{m_{t}p_{t}+n_{t}q_{t}}\left( u_{t}-\frac{m_{t}+p_{t}}{2}%
\right) dL_{t}^{X}
\end{equation*}%
and%
\begin{equation*}
dv_{t}=\frac{2}{m_{t}p_{t}+n_{t}q_{t}}\left( v_{t}-\frac{n_{t}+q_{t}}{2}%
\right) dL_{t}^{X}
\end{equation*}
\end{lemma}

\begin{proof} Applying the It\^{o} formula to the $C^{2}$ function
$f\left( m,n,p,q\right)
=\frac{m}{mp+nq}$ and processes $m_{t},n_{t},p_{t}$ and $q_{t}$, we have%
\begin{eqnarray*}
\frac{1}{2}du_{t}&=&d\left( \frac{m_{t}}{m_{t}p_{t}+n_{t}q_{t}}\right) \\
&=&\frac{1}{\left( m_{t}p_{t}+n_{t}q_{t}\right) ^{2}}\left(
n_{t}q_{t}dm_{t}-m_{t}q_{t}dn_{t}-m_{t}^{2}dp_{t}-m_{t}n_{t}dq_{t}\right) \\
&&+\frac{1}{2\left( m_{t}p_{t}+n_{t}q_{t}\right) ^{3}}\left(
-2n_{t}p_{t}q_{t}d\langle m\rangle _{t}+2m_{t}q_{t}^{2}d\langle
n\rangle _{t}+2m_{t}^{3}d\langle p\rangle
_{t}+2m_{t}n_{t}^{2}d\langle q\rangle
_{t}\right) \\
&&+\frac{1}{2\left( m_{t}p_{t}+n_{t}q_{t}\right) ^{3}}\left( 2\left(
m_{t}p_{t}q_{t}-n_{t}q_{t}^{2}\right) d\langle m,n\rangle
_{t}+4m_{t}^{2}n_{t}d\langle p,q\rangle _{t}\right) .
\end{eqnarray*}

Using the relations (\ref{process m}) and (\ref{processes n, p, q})
it can be seen that the martingale part in the last expression above
reduces to
zero, and using (\ref{quadratic variations of m, n, p, q}) we obtain%
\begin{equation*}
\frac{1}{2}du_{t}=\frac{1}{\left( m_{t}p_{t}+n_{t}q_{t}\right) ^{2}}\left[ -%
\frac{\left( m_{t}+p_{t}\right) }{2}\left( n_{t}q_{t}-m_{t}^{2}\right) +m_{t}%
\frac{\left( n_{t}+q_{t}\right) ^{2}}{2}\right] dL_{t}^{X}.
\end{equation*}

Using the fact that $L_{t}^{X}$ increases only when $X_{t}\in
\partial U$, that is only when $\left\vert X_{t}\right\vert
^{2}=\frac{\left( n_{t}+q_{t}\right) ^{2}}{4}+\frac{\left(
m_{t}+p_{t}\right) ^{2}}{4}=1$, we
obtain%
\begin{eqnarray*}
\frac{1}{2}du_{t} &=&\frac{1}{\left( m_{t}p_{t}+n_{t}q_{t}\right)
^{2}}\left[ 2m_{t}-\frac{m_{t}+p_{t}}{2}\left(
m_{t}p_{t}+n_{t}q_{t}\right) \right]
dL_{t}^{X} \\
&=&\frac{1}{m_{t}p_{t}+n_{t}q_{t}}\left(
u_{t}-\frac{m_{t}+p_{t}}{2}\right) dL_{t}^{X}.
\end{eqnarray*}

A similar proof gives the expression for the process $v_{t}$,
concluding the proof. \end{proof}

Next, we prove that the mirror coupling in the unit disk leaves
invariant the \textquotedblleft left\textquotedblright\ and
\textquotedblleft right\textquotedblright\ positions of the starting
points of the coupling, as follows:

\begin{lemma}
\label{mirror moves left}If $X_{t},Y_{t}$ is a mirror coupling of
reflecting
Brownian motions in $U$ with starting points $X_{0}=x+r$, $%
Y_{0}=x+re^{i\theta }$, $x\in \left( 0,1\right) $, $r\in \left(
0,\min \left\{ x,1-x\right\} \right) $ and $\theta \in \lbrack
0,2\pi )$, then for all time $t<\tau \wedge \tau _{1}$ the mirror
$M_{t}$ moves away from the point $x$, in such a way that
$a_{t}^{1}$ and $b_{t}^{1}$ are non-increasing
functions of $t$, where $A_{t}=\left( a_{t}^{1},a_{t}^{2}\right) $ and $%
B_{t}=\left( b_{t}^{1},b_{t}^{2}\right) $ are the points of intersection of $%
M_{t}$ with the boundary of $U$ (see Figure \ref{mirror coupling
figure}).
\end{lemma}

\begin{proof} Assuming the contrary, there exists a point $P=\left(
\alpha ,\beta \right) \in U$ and times $0<t_{1}<t_{2}<\tau \wedge
\tau _{1}$ such that $P$ is to the
right of $M_{t_{1}}$ and to the left of $M_{t_{2}}$, that is we have%
\begin{equation*}
\left\{
\begin{array}{c}
m_{t_{1}}\alpha +n_{t_{1}}\beta -\frac{1}{2}\left(
m_{t_{1}}p_{t_{1}}+n_{t_{1}}q_{t_{1}}\right) >0 \\
m_{t_{2}}\alpha +n_{t_{2}}\beta -\frac{1}{2}\left(
m_{t_{2}}p_{t_{2}}+n_{t_{2}}q_{t_{2}}\right) <0%
\end{array}%
\right.
\end{equation*}%
or equivalent%
\begin{equation*}
\left\{
\begin{array}{c}
\alpha u_{t_{1}}+\beta v_{t_{1}}>1 \\
\alpha u_{t_{2}}+\beta v_{t_{2}}<1%
\end{array}%
\right.
\end{equation*}

Setting $t_{0}=\inf \left\{ t>t_{1}:\alpha u_{t}+\beta
v_{t}<1\right\} \in
\left( t_{1},t_{2}\right) $, from the previous lemma we obtain%
\begin{eqnarray*}
&&\alpha u_{t_{0}}+\beta v_{t_{0}} \\
&=&\alpha u_{t_{1}}+\beta v_{t_{1}}+\int_{t_{1}}^{t_{0}}\frac{2}{%
m_{t}p_{t}+n_{t}q_{t}}\left( \alpha u_{t}+\beta v_{t}-\left( \frac{%
m_{t}+p_{t}}{2}\alpha +\frac{n_{t}+q_{t}}{2}\beta \right) \right)
dL_{t}^{X}
\\
&\geq &\alpha u_{t_{1}}+\beta v_{t_{1}} \\
&>&1,
\end{eqnarray*}%
since $\alpha u_{t}+\beta v_{t}\geq 1$ for $t\in \lbrack t_{1},t_{0}]$ and%
\begin{equation*}
\left\vert \left( \frac{m_{t}+p_{t}}{2}\alpha
+\frac{n_{t}+q_{t}}{2}\beta \right) \right\vert =\left\vert \alpha
X_{t}^{1}+\beta X_{t}^{2}\right\vert \leq \left\vert
X_{t}\right\vert \left\vert \left( \alpha ,\beta \right) \right\vert
<1.
\end{equation*}

By the continuity of the processes $u_{t}$ and $v_{t}$ we also must have $%
\alpha u_{t_{0}}+\beta v_{t_{0}}=1$, a contradiction, which proves
the claim. \end{proof}

From the previous lemma we obtain the following:

\begin{theorem}
\label{theorem 1}For any $x\in \left( 0,1\right) $, $r,\varepsilon
\in \left( 0,\min \left\{ x,1-x\right\} \right) $ and $\theta \in
\lbrack 0,2\pi
)$%
\begin{equation}
E^{x+re^{i\theta }}1_{B\left( x,\varepsilon \right) }\left(
Y_{t}\right) \leq E^{x+r}1_{B\left( x,\varepsilon \right) }\left(
X_{t}\right) ,\qquad t\geq 0.
\end{equation}
\end{theorem}

\begin{proof} Let $X_{t},Y_{t}$ be a mirror coupling of reflecting
Brownian motions in $U$ with starting points $X_{0}=x+r$,
$Y_{0}=x+re^{i\theta }$, $x\in \left( 0,1\right) $, $r,\varepsilon
\in \left( 0,\min \left\{ x,1-x\right\} \right) $ and $\theta \in
\lbrack 0,2\pi )$. From the previous lemma it follows that the
mirror $M_{t}$ does not separate the points $x$ and $X_{t}$ (hence
it separates the points $x$ and $Y_{t}$) for $t<\tau \wedge \tau
_{1}$.

Since for $t\geq \tau \wedge \tau _{1}$ either the two processes
$X_{t}$ and $Y_{t}$ are symmetric with respect to the (fixed) line
$M_{\tau \wedge \tau _{1}}$ passing through the origin (for $t\in
\left( \tau \wedge \tau _{1},\tau\right) $) or they have coupled
(for $t\in \left( \tau,\infty \right) $), it follows that in this
case $Y_{t}$ cannot reach $B\left( x,\varepsilon \right)$ before
coupling with $X_{t}$, and combining with the previous case we
obtain that for all times $t \ge 0$, $Y_{t}\in B\left( x,\varepsilon
\right) $ implies $X_{t}\in B\left( x,\varepsilon \right)$, and the
claim follows.
\end{proof}

Denoting by $p_{U}(t,x,y)$ the Neumann heat kernel for the unit disk
(or equivalently, the transition density of reflecting Brownian
motion in $U$), we have the following:

\begin{lemma}
\label{lema1} For any $t>0$ and $x,y\in U$ we have the following
representation formula%
\begin{equation}
p_{U}(t,x,y)=\underset{\varepsilon \rightarrow 0}{\lim
}\frac{1}{\left\vert B(y,r)\right\vert }E^{x}1_{B(y,\varepsilon
)}(W_{t}),
\end{equation}%
where $W_{t}$ is a reflecting Brownian motion in the unit disk $U$
with $W_{0}=x$ and $E^{x}$ denotes the corresponding expectation
with respect to the probability measure $P^{x}$.
\end{lemma}

\begin{proof} Follows from the continuity of $p_{U}\left(
t,x,y\right) $, see for example \cite{PascuPascu} for a proof.
\end{proof}

Combining Theorem \ref{theorem 1} and Lemma \ref{lema1}, we obtain

\begin{proposition}
\label{extremum of transition density on circle}For any $x\in \left(
0,1\right) $, $r\in \left( 0,\min \left\{ x,1-x\right\} \right) $
and $\theta \in \lbrack 0,2\pi )$ we have%
\begin{equation}
p_{U}\left( t,x+re^{i\theta },x\right) \leq p_{U}\left(
t,x+r,x\right) ,\qquad t > 0.  \label{monot of transition densities}
\end{equation}
\end{proposition}

\begin{remark}
The previous inequality can be interpreted as an extremal property
of Brownian motion as follows:%
\begin{equation*}
\max_{y\in \partial B\left( x,r\right) }p_{U}\left( t,x,y\right)
=p_{U}\left( t,x,x+r\right) ,\qquad t > 0,
\end{equation*}%
that is, among all reflecting Brownian motions in the unit disk with
starting points on the circle $\partial B\left( x,r\right) $, the
Brownian motion starting closest to the boundary (i.e. at the point
$x+r$) is most likely to return to (a neighborhood of) $x$. We will
see that this property will allow us to prove the desired
monotonicity in the Laugesen-Morpurgo conjecture.
\end{remark}

We can now prove the main monotonicity property, as follows:

\begin{theorem}
\label{main theorem}For any $x\in \left( 0,1\right) $ and $r\in
\left(
0,\min \left\{ x,1-x\right\} \right) $ we have%
\begin{equation}
\tfrac{1}{2\pi }\int_{0}^{2\pi }p_{U}\left( t,x+re^{i\theta
},x\right) d\theta \leq p_{U}\left( t,x+r,x\right) \leq p_{U}\left(
t,x+r,x+r\right) ,\quad t > 0.  \label{main inequality}
\end{equation}
\end{theorem}

\begin{proof} The first inequality follows by integrating the
inequality (\ref{monot of transition densities}) with respect to
$\theta \in [0, 2\pi)$.

The second inequality follows by an argument similar to the one in
Proposition \ref{extremum of transition density on circle} (see for
example \cite{PascuPascu}). \end{proof}

As a corollary of the above theorem, we obtain the following:

\begin{theorem}[Laugesen-Morpurgo conjecture]
\label{Laugesen-Morpurgo conjecture}The diagonal element of the
Neumann heat kernel of the unit disk is a radially increasing function, that is%
\begin{equation*}
p_{U}\left( t,x,x\right) <p_{U}\left( t,y,y\right) ,
\end{equation*}%
for any $t > 0$ and any $x,y\in \left( 0,1\right) $ with $x<y$.
\end{theorem}

\begin{proof} Let $x\in \left( 0,1\right) $ and $t>0$ be arbitrarily
fixed. From Theorem \ref{main
theorem} we obtain%
\begin{eqnarray*}
p_{U}\left( t,x+r,x+r\right) -p_{U}\left( t,x,x\right) &\geq &\frac{1}{2\pi }%
\int_{0}^{2\pi }p_{U}\left( t,x+re^{i\theta },x\right) d\theta
-p_{U}\left(
t,x,x\right) \\
&=&\frac{1}{2\pi }\int_{0}^{2\pi }p_{U}\left( t,x+re^{i\theta
},x\right) -p_{U}\left( t,x,x\right) d\theta ,
\end{eqnarray*}%
for any $r\in \left( 0,\min \left\{ x,1-x\right\} \right) $.

Dividing by $r$ and passing to the limit with $r\searrow 0$, we obtain:%
\begin{eqnarray*}
\frac{d}{dx}p_{U}\left( t,x,x\right) &=&\lim_{r\searrow
0}\frac{p_{U}\left(
t,x+r,x+r\right) -p_{U}\left( t,x,x\right) }{r} \\
&\geq &\frac{1}{2\pi }\lim_{r\searrow 0}\int_{0}^{2\pi
}\frac{p_{U}\left( t,x+re^{i\theta },x\right) -p_{U}\left(
t,x,x\right) }{r}d\theta .
\end{eqnarray*}

By bounded convergence theorem ($p_{U}\left( t,\cdot ,x\right) $ is
a $C^{2}$ function in the second variable, hence $\nabla p_{U}\left(
t,\cdot ,x\right)
$ is bounded in a neighborhood of $x$), we obtain%
\begin{equation*}
\frac{d}{dx}p_{U}\left( t,x,x\right) \geq \frac{1}{2\pi
}\int_{0}^{2\pi }\nabla p_{U}\left( t,x,x\right) \cdot e^{i\theta
}d\theta =0,
\end{equation*}%
where we denoted by $\nabla p_{U}$ the gradient of $\nabla
p_{U}\left( t,\cdot ,x\right) $ in the second variable.

Since $x\in \left( 0,1\right) $ was arbitrarily chosen, we have%
\begin{equation*}
\frac{d}{dx}p_{U}(t,x,x)\geq 0,\qquad x\in \left( 0,1\right) ,
\end{equation*}%
which shows that $p_{U}\left( t,x,x\right) $ is an increasing
function of $x\in \left( 0,1\right) $ for any $t>0$ arbitrarily
fixed.

Since $p_{U}\left( t,x,x\right) $ is the diagonal of a heat kernel
of an operator with real analytic coefficients implies that
$p_{U}\left( t,x,x\right) $ is a real analytic function, and
therefore it cannot be constant on a non-empty subset of $\left(
0,1\right) $. This, together with the fact that $p_{U}\left(
t,x,x\right) $ is increasing on $\left( 0,1\right) $ shows that
$p_{U}\left( t,x,x\right) $ is in fact strictly increasing for $x\in
\left( 0,1\right) $ for any $t>0$ arbitrarily fixed, concluding the
proof. \end{proof}

\section{Extensions and Applications\label{Extensions and Applications}}

The Laugesen-Morpurgo Conjecture \ref{Laugesen-Morpugo conjecture in
R^2} has a natural extension to $\mathbb{R}^{n}$ for any $n\in
\mathbb{N}^{\ast }$, as follows:

\begin{conjecture}[General Laugesen-Morpurgo conjecture]
For $n\in \mathbb{N}^{\ast }$ arbitrarily fixed, let
$p_{\mathbb{U}}(t,x,y)$ denote the heat kernel for the Laplacian
with Neumann boundary conditions on the unit ball
$\mathbb{U}=\left\{ x\in \mathbb{R}^{n}:\left\vert \left\vert
x\right\vert \right\vert <1\right\} $ in $\mathbb{R}^{n}$. For any
$t>0$ arbitrarily fixed, the radial function
$p_{\mathbb{U}}(t,x,x)$, called the
diagonal element of the heat kernel, is a strictly increasing function of $%
\left\vert |x|\right\vert $, that is,
\begin{equation}
p_{\mathbb{U}}(t,x,x)<p_{\mathbb{U}}(t,y,y), \label{general
laugesen-Morpurgo conjecture}
\end{equation}%
for any $t>0$ and $x,y\in \mathbb{U}$ with $\left\vert \left\vert
x\right\vert \right\vert <\left\vert \left\vert y\right\vert
\right\vert $.
\end{conjecture}

We can use the ideas of the previous section (the case $n=2$) to
prove the above conjecture for any $n\in \mathbb{N}^{\ast }$, as
follows.

In the case $n=1$, for $0<x<1$ and $0<r<\min \left\{ x,1-x\right\} $
arbitrarily fixed, the mirror coupling with starting points $x-r$
and $x+r$
given by (\ref{mirror coupling in U})-(\ref{definition of Z_t}) becomes%
\begin{equation}
\left\{
\begin{array}{l}
X_{t}=x-r+W_{t}+\int_{0}^{t}\nu \left( X_{s}\right) dL_{s}^{X} \\
Y_{t}=x+r-W_{t}+\int_{0}^{t}\nu \left( Y_{s}\right) dL_{s}^{Y}%
\end{array}%
\right.
\end{equation}%
where $W_{t}$ is a $1$-dimensional Brownian motion starting at
$W_{0}=0.$

Denoting by $M_{t}$ the midpoint between $X_{t}$ and $Y_{t}$, we have%
\begin{equation*}
M_{t}=\frac{X_{t}+Y_{t}}{2}=x+\frac{1}{2}\int_{0}^{t}\nu \left(
X_{s}\right) dL_{s}^{X}+\frac{1}{2}\int_{0}^{t}\nu \left(
Y_{s}\right) dL_{s}^{Y},
\end{equation*}%
and since $L_{s}^{X}$ is constant for $s<\tau _{1}=\inf \left\{
s>0:X_{s}=-1\right\} $, and for $s<\tau =\inf \left\{
s>0:X_{s}=Y_{s}\right\} $ the process $L_{s}^{Y}$ can increase only
if $\nu \left( Y_{s}\right) =\nu \left( 1\right) =-1$ (the process
$Y_{s}$ cannot hit the boundary point $-1$ unless it couples with
$X_{s}$ first), it follows that%
\begin{equation*}
M_{t\wedge \tau \wedge \tau _{1}}=x-\frac{1}{2}L_{t\wedge \tau
\wedge \tau _{1}}^{Y},\qquad t\geq 0,
\end{equation*}%
which shows that $M_{t}$ is decreasing on the interval $\left(
0,\tau \wedge \tau _{1}\right) $, and therefore the process $Y_{t}$
is closer to $x$ than $X_{t}$ for all $t\in \left( 0,\tau \wedge
\tau _{1}\right) $.

Since for $t\geq $ $\tau \wedge \tau _{1}$ either the processes $X_{t}$ and $%
Y_{t}$ move together (for $t\in [ \tau ,\infty )$), or they are
symmetric with respect to the origin (for $t\in [ \tau \wedge \tau
_{1},\tau ]$), it follows that for all $t> 0$ the process $Y_{t}$ is
always closer to the point $x$ than the process $X_{t}$, and
therefore we obtain%
\begin{equation*}
p_I\left( t,x-r,x\right) \leq p_I\left( t,x+r,x\right) ,
\end{equation*}%
for all $t>0,$ $x\in \left( 0,1\right) $ and $r\in \left( 0,\min
\left\{ x,1-x\right\} \right) $, where $p_I\left( t,x,y\right) $
denotes the transition density of reflecting Brownian motion on the
interval $I=\left( -1,1\right) $.

A similar proof shows that $p_I\left( t,x+r,x\right) \leq p_I\left(
t,x+r,x+r\right) $, and therefore we
obtain%
\begin{equation*}
p_I\left( t,x-r,x\right) \leq p_I\left( t,x+r,x\right) \leq
p_I\left( t,x+r,x+r\right) ,
\end{equation*}%
or equivalent%
\begin{equation*}
\frac{p_I\left( t,x-r,x\right) +p_I\left( t,x+r,x\right) }{2} \leq
p_I\left( t,x+r,x\right) \leq p_I\left( t,x+r,x+r\right) ,
\end{equation*}%
for all $t>0$, $x\in \left( 0,1\right) $ and $r\in \left( 0,\min
\left\{ x,1-x\right\} \right) $, which corresponds to the double
inequality (\ref{main inequality}) in the case $n=1$.

Following the proof of Theorem \ref{Laugesen-Morpurgo conjecture},
for $x\in \left( 0,1\right) $ arbitrarily
fixed we obtain%
\begin{eqnarray*}
\frac{\partial }{\partial x}p_I\left( t,x,x\right)
&=&\lim_{r\searrow 0}\frac{p_I\left( t,x+r,x+r\right)
-p_I\left( t,x,x\right) }{r} \\
&\geq &\frac{1}{2}\lim_{r\searrow 0}\frac{p_I\left(
t,x-r,x\right) -p_I\left( t,x,x\right) }{r} \\
&&+\frac{p_I\left( t,x+r,x\right) -p_I\left( t,x,x\right) }{r} \\
&=&\frac{1}{2}\left( -p_I^{\prime }\left( t,x,x\right)
+p_I^{\prime }\left( t,x,x\right) \right)  \\
&=&0,
\end{eqnarray*}
where we denoted by $p_I^{\prime }$ the derivative of the function
$p_I\left( t,\cdot ,x\right) $ in the second variable, concluding
the proof in the $1$-dimensional case.

\begin{remark}
The fact that the Laugesen-Morpurgo conjecture is true in the case
$n=1$ is known (see for example \cite{Banuelos}, Remark 5.4 for an
analytic proof, or \cite{PascuPascu1} for a different probabilistic
proof).We presented it here for a unitary treatment of the
Laugesen-Morpurgo conjecture, in order to show that the same
argument can be applied regardless of the dimension $n\in
\mathbb{N}^{\ast }$.
\end{remark}

The coupling arguments in the previous section can also be applied,
with the appropriate changes, to give a proof of the general
Laugesen-Morpurgo Conjecture \ref{general laugesen-Morpurgo
conjecture} in the case $n\geq 3$. For example, in the case $n=3$,
replacing in (\ref{mirror coupling in U})-(\ref{definition of Z_t})
$W_{t}$ by a $3$-dimensional Brownian motion starting at $W_{0}=0$
and $\nu $ by the the corresponding unit normal vector field on the
boundary of the unit ball $\mathbb{U}=\left\{ x\in
\mathbb{R}^{3}:\left\vert \left\vert
x\right\vert \right\vert <1\right\} $, $\nu _{\mathbb{U}}\left( x\right) =-x$%
, where $x=\left( x^{1},x^{2},x^{3}\right) \in \partial \mathbb{U}$,
the
same formulae give the mirror coupling in $\mathbb{U}$ with starting points $%
X_{0}=x\in \mathbb{U}$ and $Y_{0}=y\in \mathbb{U}$.

Following Lemma \ref{mirror moves left}, for distinct starting points $%
X_{0}=\left( x^{1}+\rho ,0,0\right) ,Y_{0}=\left( x^{1},0,0\right)
+\rho u\in \mathbb{U}$ with $x^{1}\in \left( 0,1\right) $, $\rho \in
\left( 0,\min \left\{ 1-x^{1},x^{1}\right\} \right) $ and
$\left\vert \left\vert u\right\vert \right\vert =1$, we need to show
that for $t\leq \tau \wedge \tau _{1}$ ($\tau $ and $\tau _{1}$ are
the coupling time, respectively the hitting time of the origin by
$M_{t}$), the plane of symmetry $M_{t}$
between between $X_{t}$ and $Y_{t}$ given by%
\begin{equation*}
\left( X_{t}^{1}-Y_{t}^{1}\right) z^{1}+\left(
X_{t}^{2}-Y_{t}^{2}\right) z^{2}+\left( X_{t}^{3}-Y_{t}^{3}\right)
z^{3}-\frac{\left\vert \left\vert X_{t}\right\vert \right\vert
^{2}-\left\vert \left\vert Y_{t}\right\vert \right\vert ^{2}}{2}
=0,\quad \left( z^{1},z^{2},z^{3}\right) \in \mathbb{R}^{3},
\end{equation*}%
moves away from the point $M_{0}=x$, towards the origin (i.e. to the
\textquotedblleft left\textquotedblright ).

Assuming the contrary, there exists a point $P=\left( \alpha ,\beta
,\gamma \right) \in \mathbb{U}$ and times $0<t_{1}<t_{2}<\tau \wedge
\tau _{1}$ such that $P$ is to the right of $M_{t_{1}}$ and to the
left of $M_{t_{2}}$, that is we
have:%
\begin{equation*}
\left\{
\begin{array}{c}
\left( X_{t_{1}}^{1}-Y_{t_{1}}^{1}\right) \alpha +\left(
X_{t_{1}}^{2}-Y_{t_{1}}^{2}\right) \beta +\left(
X_{t_{1}}^{3}-Y_{t_{1}}^{3}\right) \gamma -\frac{1}{2}\left(
\left\vert \left\vert X_{t_{1}}\right\vert \right\vert
^{2}-\left\vert \left\vert
Y_{t_{1}}\right\vert \right\vert ^{2}\right) >0 \\
\left( X_{t_{2}}^{1}-Y_{t_{2}}^{1}\right) \alpha +\left(
X_{t_{2}}^{2}-Y_{t_{2}}^{2}\right) \beta +\left(
X_{t_{2}}^{3}-Y_{t_{2}}^{3}\right) \gamma -\frac{1}{2}\left(
\left\vert \left\vert X_{t_{2}}\right\vert \right\vert
^{2}-\left\vert \left\vert
Y_{t_{2}}\right\vert \right\vert ^{2}\right) <0%
\end{array}%
\right.
\end{equation*}%
or equivalent%
\begin{equation*}
\left\{
\begin{array}{c}
\alpha u_{t_{1}}+\beta v_{t_{1}}+\gamma w_{t_{1}}>1 \\
\alpha u_{t_{2}}+\beta v_{t_{2}}+\gamma w_{t_{2}.}<1%
\end{array}%
\right. ,
\end{equation*}%
where%
\begin{equation*}
u_{t}=\frac{2\left( X_{t}^{1}-Y_{t}^{1}\right) }{\left\vert
\left\vert X_{t}\right\vert \right\vert ^{2}-\left\vert \left\vert
Y_{t}\right\vert \right\vert ^{2}},\quad v_{t}=\frac{2\left(
X_{t}^{2}-Y_{t}^{2}\right) }{\left\vert \left\vert X_{t}\right\vert
\right\vert ^{2}-\left\vert \left\vert Y_{t}\right\vert \right\vert
^{2}},\quad w_{t}=\frac{2\left( X_{t}^{3}-Y_{t}^{3}\right)
}{\left\vert \left\vert X_{t}\right\vert \right\vert ^{2}-\left\vert
\left\vert Y_{t}\right\vert \right\vert ^{2}}.
\end{equation*}

As in Lemma \ref{u and v proc of bdd var}, $u_{t}$, $v_{t}$ and
$w_{t}$ are processes of bounded variation on $\left( 0,\tau \wedge
\tau _{1}\right) $,
explicitly given by%
\begin{equation*}
du_{t}=\frac{2}{\left\vert \left\vert X_{t_{1}}\right\vert
\right\vert ^{2}-\left\vert \left\vert Y_{t_{1}}\right\vert
\right\vert ^{2}}\left( u_{t}-X_{t}^{1}\right) dL_{t}^{X},
\end{equation*}%
where $L_{t}^{X}$ is the local time of $X_{t}$ on the boundary of
$\mathbb{U} $, and similarly for $v_{t}$ and $w_{t}$.

Setting $t_{0}=\inf \left\{ t>t_{1}:\alpha u_{t}+\beta v_{t}+\gamma
w_{t}<1\right\} \in \left( t_{1},t_{2}\right) $, we obtain%
\begin{eqnarray*}
\alpha u_{t_{0}}+\beta v_{t_{0}} + \gamma w_{t_0}&=&\alpha
u_{t_{1}}+\beta v_{t_{1}}+ \gamma w_{t_1}\\
&&\qquad +2\int_{t_{1}}^{t_{0}}\frac{\alpha u_{t}+\beta v_{t}+\gamma
w_{t}-\left( \alpha X_{t}^{1}+\beta X_{t}^{2}+\gamma X_{t}^{3}\right) }{%
\left\vert \left\vert X_{t_{1}}\right\vert \right\vert
^{2}-\left\vert
\left\vert Y_{t_{1}}\right\vert \right\vert ^{2}}dL_{t}^{X} \\
&\geq &\alpha u_{t_{1}}+\beta v_{t_{1}}+ \gamma w_{t_1} \\
&>&1,
\end{eqnarray*}%
since $\alpha u_{t}+\beta v_{t}+\gamma w_{t}\geq 1$ for $t\in
\lbrack t_{1},t_{0}]$ and $$\left\vert \alpha X_{t}^{1}+\beta
X_{t}^{2}+\gamma X_{t}^{3}\right\vert \leq \left\vert \left\vert
X_{t}\right\vert \right\vert ~\left\vert \left\vert \left( \alpha
,\beta ,\gamma \right) \right\vert \right\vert \leq \left\vert
\left\vert \left( \alpha ,\beta ,\gamma \right) \right\vert
\right\vert < 1$$ for all $t\geq 0$.

By the definition of $t_0$ and the continuity of the processes
$u_{t}$, $v_{t}$ and $w_{t}$ we must also have $\alpha
u_{t_{0}}+\beta v_{t_{0}}+\gamma w_{t_{0}}=1$, contradicting $\alpha
u_{t_{0}}+\beta v_{t_{0}} + \gamma w_{t_0}>1$, which proves the
claim.

Using the previous argument and proceeding as in Theorem
\ref{theorem 1}, we obtain%
\begin{equation}
E^{\left( x^{1},0,0\right) +\rho u}1_{B\left( \left(
x^{1},0,0\right) ,\varepsilon \right) }\left( Y_{t}\right) \leq
E^{\left( x^{1}+\rho ,0,0\right) }1_{B\left( \left( x^{1},0,0\right)
,\varepsilon \right) }\left( X_{t}\right) ,
\end{equation}%
for any $t> 0$, $x^{1}\in \left( 0,1\right) $, $\rho ,\varepsilon
\in \left( 0,\min \left\{ x^{1},1-x^{1}\right\} \right) $ and $u\in
\partial \mathbb{U}$.

Using the equivalent result in Lemma \ref{lema1} for the
$3$-dimensional
reflecting Brownian motion, we obtain%
\begin{equation*}
p_{\mathbb{U}}\left( t,\left( x^{1},0,0\right) +\rho u,\left(
x^{1},0,0\right) \right) \leq p_{\mathbb{U}}\left( t,\left(
x^{1}+\rho ,0,0\right) ,\left( x^{1},0,0\right) \right) ,
\end{equation*}%
which as in Theorem \ref{main theorem} shows that for any $t> 0$,
$x=(x^1,0,0)\in \mathbb{U}$, $x^{1}>0$ and $\rho \in \left( 0,\min
\left\{ x^{1},1-x^{1}\right\} \right) $ we have%
\begin{eqnarray*}
&&\frac{1}{4\pi}\int_{\partial \mathbb{U}}p_{\mathbb{U}}\left(
t,\left( x^{1},0,0\right) +\rho u,\left(
x^{1},0,0\right) \right) d\sigma \left( u\right)  \\
&\leq &p_{\mathbb{U}}\left( t,\left( x^{1}+\rho ,0,0\right) ,\left(
x^{1},0,0\right) \right)  \\
&\leq &p_{\mathbb{U}}\left( t,\left( x^{1}+\rho ,0,0\right) ,\left(
x^{1}+\rho ,0,0\right) \right) ,
\end{eqnarray*}
where $d\sigma \left( u\right) $ is the surface measure on $\partial \mathbb{%
U}$.

Proceeding as in the proof of Theorem \ref{Laugesen-Morpurgo
conjecture} and using the rotational invariance of reflecting
Brownian motion in $\mathbb{U}$, we obtain%
\begin{equation*}
p_{\mathbb{U}}\left( t,x,x\right) <p_{\mathbb{U}}\left( t,y,y\right)
,
\end{equation*}%
for any $t>0$ and any $x,y\in \mathbb{U}$ with $\left\vert
\left\vert x\right\vert \right\vert <\left\vert \left\vert
y\right\vert \right\vert $, concluding the proof of the conjecture
in the case $n=3$.

The above discussion can be summarized in the following resolution
of the general Laugesen-Morpurgo conjecture:

\begin{theorem}
\label{general Laugesen-Morpurgo conjecture}The diagonal element $p_{\mathbb{%
U}}(t,x,x)$ of the Neumann heat kernel of the unit ball $\mathbb{U}$ in $%
\mathbb{R}^{n}$ ($n\in \mathbb{N}^{\ast }$) is a radially increasing
function, that is we have%
\begin{equation*}
p_{\mathbb{U}}\left( t,x,x\right) <p_{\mathbb{U}}\left( t,y,y\right)
,
\end{equation*}%
for any $t>0$ and any $x,y\in \mathbb{U}$ with $\left\vert
\left\vert x\right\vert \right\vert <\left\vert \left\vert
y\right\vert \right\vert $.
\end{theorem}

The problem of monotonicity of Neumann heat kernel is closely
related to the celebrated ``Hot Spots'' conjecture of Jeffrey Rauch,
which asserts that the extrema of second Neumann eigenfunctions of a
planar convex domain are attained only on the boundary of the
domain. The Hot Spots Conjecture conjecture has been verified for
several types of convex domains in the plane (see for example
\cite{Pascu2} and the references cited therein).

Using and eigenfunction expansion of the transition density of
reflecting Brownian motion, it can be seen that the monotonicity in
the Laugesen--Morpurgo conjecture implies a similar monotonicity of
the second Neumann eigenfunction(s). As an application of Theorem
\ref{Laugesen-Morpurgo conjecture}, we can derive the Hot Spots
conjecture in the case of the unit disk (see for example
\cite{Pascu2} for an alternate proof), as follows:

\begin{theorem}
\label{Hot Spots for the unit disk}If $\varphi $ is a second Neumann
eigenfunction of the Laplaceian in the unit disk, then $\varphi
\left( x\right) $ is a radially monotonic function. In particular,
$\varphi $
attains its maximum and minimum only on the boundary of $U$, that is%
\begin{equation*}
\min_{\partial U}\varphi <\varphi \left( x\right) <\max_{\partial
U}\varphi ,\qquad x\in U,
\end{equation*}%
which shows that the Hot Spots conjecture holds for the unit disk
$U$.
\end{theorem}

\begin{proof} Using an eigenfunction expansion of $p_{U}\left(
t,x,y\right) $, it can be
seen that for large $t$ we have%
\begin{equation*}
p_{U}\left( t,x,x\right) \approx \frac{1}{\sqrt{\pi }}+e^{-\lambda
_{2}t}\varphi _{2}^{2}\left( x\right) +e^{-\lambda _{2}t}\psi
_{2}^{2}\left( y\right) ,
\end{equation*}%
where $\varphi _{2}\left( re^{i\theta }\right) =J_{2}\left(
\sqrt{\lambda _{2}}r\right) \cos \theta $ and $\psi _{2}\left(
re^{i\theta }\right) =J_{2}\left( \sqrt{\lambda _{2}}r\right) \sin
\theta $ are two independent second Neumann eigenfunctions for the
Laplaceian on $U$, $\lambda _{2}$ is the second Neumann eigenvalue
and $J_{2}$ is the Bessel function of order $2$ (see for example
\cite{Bandle}, pp. 92 -- 93).

From Theorem \ref{Laugesen-Morpurgo conjecture} it follows that%
\begin{equation*}
\varphi _{2}^{2}\left( r\right) +\psi _{2}^{2}\left( r\right)
=J_{2}^{2}\left( \sqrt{\lambda _{2}}r\right)
\end{equation*}%
is an increasing function of $r$, and therefore for an arbitrary
second
Neumann eigenfunction $\varphi $ (the second Neumann eigenspace is $2$%
-dimensional) we have%
\begin{eqnarray*}
\varphi \left( re^{i\theta }\right)  &=&\alpha \varphi _{2}\left(
re^{i\theta }\right) +\beta \psi _{2}(re^{i\theta }) \\
&=&J_{2}\left( \sqrt{\lambda _{2}}r\right) \left( \alpha \cos \theta
+\beta \sin \theta \right)
\end{eqnarray*}%
is a monotonic function of $r$ for $\theta \in \lbrack 0,2\pi )$
arbitrarily fixed, and the claim follows. \end{proof}

\begin{remark}
\label{Hot Spots remark}The fact that the Hot Spots conjecture holds
in the case of the unit disk is a known result (see for example
\cite{Kawohl}, \cite{Pascu2}). The above proof is meant to show the
connection of the Laugesen-Morpurgo conjecture with the Hot Spots
conjecture, connection which may be used for a possibly different
approach in the resolution of the later conjecture.

More precisely, if the Laugesen-Morpurgo conjecture can be extended
to a certain smooth convex domain $D$, that is, if it can be shown
that the diagonal element $p_{D}\left( t,x,x\right) $ of the
transition density of the reflecting Brownian motion in $D$ is an
increasing function of $x$ along a certain family of curves covering
$D$, then, at least in the case of the $1$-dimensional second
Neumann eigenspace (it is known that the second Neumann eigenspace
is either $1$ or $2$-dimensional), as in the above proof it follows
that the second Neumann eigenfunctions are monotone along the same
family of curves, thus proving that the Hot Spots conjecture holds
for the domain $D$.
\end{remark}

\bibliographystyle{amsalpha}

\end{document}